\documentclass{amsart}

\usepackage{amsmath,latexsym,amssymb,amsthm,graphicx}
\usepackage{subfigure}
\usepackage{xfrac}
\usepackage{fix-cm}
\usepackage{hyperref,doi}
\usepackage{stmaryrd}
\usepackage{multirow}
\usepackage{tabularx}
\usepackage[all]{xy}
\usepackage{enumitem}
\usepackage{color}

\setenumerate[1]{label=(\thesection.\arabic*)}

\numberwithin{equation}{section}
\newtheorem{theorem}{Theorem}[section]
\newtheorem*{maintheorem}{Main Theorem}
\newtheorem{lemma}[theorem]{Lemma}

\theoremstyle{definition}
\newtheorem{remark}[theorem]{Remark}

\newtheorem{example}[theorem]{Example}

\def\R{\ensuremath{\mathbb{R}}}

\newcommand{\pa}[1]{\left(#1\right)}
\newcommand{\cpa}[1]{\left\{#1\right\}}

\newcommand{\tn}[1]{\textnormal{#1}}
\newcommand{\br}[1]{\left[#1\right]}

\newcommand{\Int}[1]{\tn{Int} \, #1}

\font\cuf=cmtt8
\newcommand{\curl}[1]{{\cuf #1}}

\begin{document}
\title{Double branched covers of theta-curves}

\author[J.~Calcut]{Jack S. Calcut}
\address{Department of Mathematics\\
         Oberlin College\\
         Oberlin, OH 44074}
\email{jcalcut@oberlin.edu}
\urladdr{\href{http://www.oberlin.edu/faculty/jcalcut/}{\curl{http://www.oberlin.edu/faculty/jcalcut/}}}

\author[J.~Metcalf-Burton]{Jules R. Metcalf-Burton}
\address{Department of Mathematics\\
         University of Michigan\\
         Ann Arbor, MI 48109}
\email{julesmb@umich.edu}

\date{\today}

\begin{abstract}
We prove a folklore theorem of W. Thurston which provides necessary and sufficient conditions for primality of a certain class of theta-curves.
Namely, a theta-curve in the 3-sphere with an unknotted constituent knot $\kappa$ is prime if and only if
lifting the third arc of the theta-curve to the double branched cover over $\kappa$ produces a prime knot.
We apply this result to Kinoshita's theta-curve.
\end{abstract}

\maketitle

\section{Introduction}
\label{sec:introduction}

Consider the multigraph $\Gamma$ with two vertices $v_1,v_2$ and three edges $e_1,e_2,e_3$ none of which are loops.
A \emph{theta-curve} is a locally flat embedding of $\Gamma$ in $S^3$ or in $\R^3$.
Each theta-curve $\theta$ has three \emph{constituent knots}: $e_2\cup e_3$, $e_1\cup e_3$, and $e_1\cup e_2$.
Given a constituent knot $\kappa$, there is exactly one arc $e$ of $\theta$ not contained in $\kappa$.\\

A theta-curve is \emph{unknotted} provided it lies on an embedded $S^2$ and is \emph{knotted} otherwise.
We will use two operations on knots and theta-curves: the order-$2$ connected sum $\#_2$ and order-$3$ connected sum $\#_3$.
An order-2 connected sum $\theta\#_2 J$ of a theta-curve $\theta\subset S^3$ and a knot $J\subset S^3$ is the result of deleting unknotted ball-arc pairs from each of $(S^3,\theta)$ and $(S^3,J)$, and then identifying the resulting boundary spheres.
The ball-arc pair in $(S^3,\theta)$ must be disjoint from the vertices $v_1$ and $v_2$.
An order-3 connected sum $\theta_1\#_3\theta_2$ of two theta-curves is the result of deleting an unknotted ball-prong neighborhood of a vertex from each theta-curve, and then identifying the resulting boundary spheres.
Each of these operations yields a theta-curve in $S^3$.\\

\begin{remark}
Wolcott has shown that the order-$3$ connected sum is independent of the glueing homeomorphism provided one specifies the vertices at which to sum and the pairing of the arcs~\cite{Wol87}.
The operations $\theta\#_2 J$ and $\theta_1\#_3\theta_2$ are ambiguous as presented. The former could mean up to six different theta-curves and the latter could mean up to $24$ different theta-curves. In each instance below, this ambiguity is either irrelevant or is sufficiently eliminated by context.
\end{remark}

A theta-curve $\theta$ is \emph{prime} provided the following three conditions are satisfied:
(i) $\theta$ is knotted,
(ii) $\theta$ is not an order-$2$ connected sum of a nontrivial knot and a (possibly unknotted) theta-curve, and
(iii) $\theta$ is not an order-$3$ connected sum of two knotted theta-curves.
We adopt the convention that the unknot is not prime.\\

Let $S^3\subset\R^4$ be the unit sphere. Let $g\in SO(4)$ denote the orientation preserving involution of $S^3$ whose matrix is diagonal with entries $\br{-1,-1,1,1}$.
Note that $\tn{Fix}(g)$ is a great circle in $S^3$ and is therefore unknotted.
Let $G=\cpa{e,g}$ be the group of order two.
Throughout this paper, equivariance is with respect to $G$.
Recall that a subset $D\subset S^3$ is \emph{equivariant} provided $g(D)=D$ (setwise) or $g(D)\cap D = \emptyset$.\\

Suppose the theta-curve $\theta$ has an unknotted constituent knot $\kappa$ and $\theta=e\cup\kappa$.
By an ambient isotopy, we can assume $\kappa=\tn{Fix}(g)$.
Let $b:S^3\to S^3$ be the double branched cover with branch set $\tn{Fix}(g)$ and such that $b\circ g = b$.
Lifting the arc $e$ of $\theta$ yields the knot $K:=b^{-1}(e)$ in $S^3$.
We call $K$ the \emph{lifted knot of $\theta$ with respect to $\kappa$}.
By stereographic projection from $(0,0,0,1)\in S^3$, the map $g$ descends to the rotation of $\R^3$ about the $z$-axis by half a revolution.\\

Moriuchi attributes the following theorem to W.~Thurston without proof~\cite[Prop.~4.1]{Mor04},
and Moriuchi references an unpublished letter of Litherland for Thurston's statement of this theorem~\cite[Prop.~5.1]{Mor09}.

\begin{maintheorem}[W.~Thurston]
Suppose $\theta$ has an unknotted constituent knot $\kappa$, and let $K$ be the lifted knot of $\theta$ with respect to $\kappa$.
The theta-curve $\theta$ is prime if and only if the lifted knot $K$ is prime.
\end{maintheorem}

Our proof of the Main Theorem uses the equivariant Dehn Lemma in two places.
In the proof of Lemma~\ref{lem:unknotted}, Kim and Tollefson's version~\cite[Lemma~3]{KT80} suffices as noted by the referee.
In the proof of Lemma~\ref{lem:torus}, we use Edmond's version~\cite{Edm86}.
It's possible that with some tinkering Kim and Tollefson's version suffices.\\

In the final section of this paper, we use the Main Theorem to prove that Kinoshita's theta-curve is prime.
We also explain how primality of Kinoshita's theta-curve yields an alternate proof of the irreducibility of certain tangles.\\

\noindent\textbf{Acknowledgment.} We thank the referee for simplifying the proof of Lemma~\ref{lem:unknotted} and for several other helpful comments.

\section{Proof of Main Theorem}
\label{sec:proof}

We will prove the contrapositive in both directions.
First, we observe two useful lemmas.

\begin{lemma}\label{lem:sum}
Suppose $\theta$, $\theta_1$, and $\theta_2$ are theta-curves, and $J\subset S^3$ is a nontrivial knot.
\begin{enumerate}
\item \label{first_part}If $\theta\#_2J$ has an unknotted constituent knot $\kappa$, then $\kappa\subset\theta$ (that is, $\kappa$ is the union of two edges in $\theta$).
Let $K$ be the lifted knot of $\theta$ with respect to $\kappa$.
Then, the lifted knot of $\theta\#_2J$ with respect to $\kappa$ is $K\#J\#J$.
\item \label{second_part}If $\theta_1\#_3\theta_2$ has an unknotted constituent knot $\kappa$,
then there are unknotted constituent knots $\kappa_1\subset\theta_1$ and $\kappa_2\subset\theta_2$ such that the order-3 connected sum $\theta_1\#_3\theta_2$ induces the connected sum $\kappa=\kappa_1\#\kappa_2$.
Let $K_j$ be the lifted knot of $\theta_j$ with respect to $\kappa_j$ for $j\in\cpa{1,2}$.
Then, the lifted knot of $\theta_1\#_3\theta_2$ with respect to $\kappa$ is $K_1\#K_2$.
\end{enumerate}
\end{lemma}
\begin{proof}
For the first claim in~\ref{first_part}, label the edges of $\theta$ so that the sum $\theta\#_2J$ is performed along $e_3\subset\theta$ and let $e$ be the resulting edge in $\theta\#_2J$. The knot $e_1\cup e$ is the connected sum of the knots $e_1\cup e_3$ and $J$. As $J$ is nontrivial and nontrivial knots do not have inverses under connected sum, we see that the knot $e_1\cup e$ is nontrivial. Similarly, the knot $e_2\cup e$ is nontrivial. Hence, $e_1\cup e_2$ must be an unknot as desired. The first claim in~\ref{second_part} follows similarly.

The remaining claims follow from the definitions of order-2 and order-3 connected sum and from the definition of the double branched cover $b:S^3\to S^3$.
\end{proof}

For the remainder of this section, we assume $\theta$ is a theta-curve with unknotted constituent knot $\kappa=\tn{Fix}(g)$ and $\theta=e\cup\kappa$. 
Let $K$ be the lifted knot of $\theta$ with respect to $\kappa$.

\begin{lemma}\label{lem:sphere}
If $\Sigma$ is an equivariant 2-sphere in $S^3$ that meets $\kappa$ in exactly two points, then $b(\Sigma)$ is an embedded 2-sphere in $S^3$ transverse to $\kappa$ and meeting $\kappa$ in exactly two points.
\end{lemma}
\begin{proof}
As $\Sigma$ meets $\kappa=\tn{Fix}(g)$, equivariance implies $g(\Sigma)=\Sigma$.
Equivariance also implies that $\Sigma$ is transverse to $\kappa$, $b(\Sigma)$ is a closed connected surface, and $b(\Sigma)$ is transverse to $\kappa$.
Let $\Sigma'$ be the equivariant annulus obtained from $\Sigma$ by deleting the interiors of disjoint equivariant $2$-disk neighborhoods of the two points $\Sigma\cap\kappa$.
The restriction of $b$ to $\Sigma'\to b(\Sigma')$ is a double cover and $b(\Sigma)$ is obtained from $b(\Sigma')$ by glueing in two $2$-disks. It follows that the Euler characteristic of $b(\Sigma')$ is 0 and the Euler characteristic of $b(\Sigma)$ is 2. 
Hence, $b(\Sigma)$ is a $2$-sphere.
\end{proof}

\begin{lemma}\label{lem:unknotted}
The theta-curve $\theta$ is unknotted if and only if $K$ is unknotted.
\end{lemma}
\begin{proof}
Clearly, if $\theta$ is unknotted, then $K$ is unknotted.

Suppose $K$ is unknotted.
Let $N$ be a closed regular equivariant neighborhood of $K$ in $S^3$.
As $S^3-\Int N$ is a solid torus, it has compressible boundary.
By the equivariant Dehn Lemma~\cite[Lemma~3]{KT80}, there exists a properly embedded equivariant compressing disk $D\subset S^3-\Int N$.
If $g(D)=D$, then a slight pushoff $D'$ of $D$ is an equivariant compressing disk with $g(D')\cap D'=\emptyset$ and we may redefine $D$ to be $D'$ instead.
Hence, we may assume $D\cap g(D)=\emptyset$.
As $D\cap N=\partial D$ is a longitude of the solid torus $N$, there exists an embedded annulus $A\subset N$ such that $\partial A=K\cup\partial D$ and $A\cap g(A)=K$.
The $2$-sphere $\Sigma=D\cup A\cup g(A)\cup g(D)$ contains $K$, is transverse to $\kappa=\tn{Fix}(g)$, and meets $\kappa$ in exactly two points.
Therefore, $\Sigma$ divides $\pa{S^3,\kappa}$ into two equivariant unknotted ball-arc pairs.
Each of these balls contains an equivariant $2$-disk with boundary $K$ and containing that balls arc of $\kappa$.
The union of these two disks is a new $2$-sphere $\Sigma'$ such that $b(\Sigma')$ is a sphere containing $\theta$ as desired.
\end{proof}

To prove the reverse implication of the Main Theorem,
suppose $\theta$ is not prime.
If $\theta$ is unknotted, then $K$ is the unknot which is not prime.
If $\theta=\theta_0\#_2J$ and $J$ is nontrivial, then, by Lemma~\ref{lem:sum}, $K=K_0\#J\#J$ which is not prime.
Otherwise, $\theta=\theta_1\#_3\theta_2$, where $\theta_1$ and $\theta_2$ are both knotted.
Then, by Lemma~\ref{lem:sum}, $K=K_1\#K_2$ is a sum of knots which are nontrivial by Lemma~\ref{lem:unknotted}. This proves the reverse implication of the Main Theorem.\\

To prove the forward implication of the Main Theorem, suppose $K$ is not prime. If $K$ is unknotted, then so is $\theta$ by Lemma~\ref{lem:unknotted}. Otherwise, there is a sphere $\Sigma$ that splits $(S^3,K)$ into two knotted ball-arc pairs.

\begin{lemma}\label{lem:connect_sum}
If $\Sigma\cap g(\Sigma)=\emptyset$, then $\theta$ is a nontrivial order-2 connected sum.
If $\Sigma=g(\Sigma)$ and $\Sigma$ meets $\kappa$ at exactly two points distinct from $v_1$ and $v_2$, then $\theta$ is a nontrivial order-3 connected sum.
\end{lemma}
\begin{proof}
If $\Sigma\cap g(\Sigma)=\emptyset$, then $\Sigma$ bounds a ball $B$ disjoint from $g(\Sigma)$. 
It follows that $B$ is disjoint from $g(B)$.
As $\kappa=\tn{Fix}\pa g$ is connected and disjoint from $\Sigma$, $\kappa$ must also be disjoint from $B$. Thus $(B,B\cap K)$ maps homeomorphically by $b$ to a nontrivial ball-arc pair in $(S^3,\theta)$.
Thus, $\theta$ is a nontrivial order-2 connected sum.

Suppose $\Sigma=g(\Sigma)$ meets $\kappa$ at exactly two points distinct from $v_1$ and $v_2$. 
As $g(\Sigma\cap K)=\Sigma\cap K$ and $g$ interchanges the two lifts $\alpha$ and $\beta$ of $e$, $\Sigma$ meets each of $\alpha$ and $\beta$ once.
Therefore, the vertices $v_1$ and $v_2$ must lie in different components of $S^3-\Sigma$, and so $\Sigma$ meets each arc of $\kappa$ once.
In particular, the $G$-action does not interchange the balls in $S^3$ bounded by $\Sigma$.
%Let $c$ be an invariant simple closed curve in $\Sigma$ containing two points of the exceptional set (why does this exist?).
%Then $c$ bounds a disk $D\subset\Sigma$ and either $D=g(D)$ or $D\cup g(D)=\Sigma$.
%If $D=g(D)$ then it cannot be transverse to the exceptional set (why?), so $D\cup g(D)=\Sigma$.
%Then $D$ only hits the exceptional set on its boundary, so $\Sigma$ hits the exceptional set at exactly two points.
By Lemma~\ref{lem:sphere}, $b(\Sigma)$ is a sphere and it splits $\theta$ as an order-3 connected sum.
Each of these summands must be nontrivial since $\Sigma$ splits $(S^3,K)$ into two knotted ball-arc pairs.
\end{proof}

Let $\Sigma$ be a sphere such that:
\begin{enumerate}
\item\label{2.1} $\Sigma$ separates $(S^3,K)$ into two knotted ball-arc pairs.
\item\label{transverse} $\Sigma$ and $g(\Sigma)$ are in general position with each other.
\item\label{2.3} $\Sigma\cap g(\Sigma)\cap K=\emptyset$.
\end{enumerate}
Condition~\ref{transverse} is achieved by the proof of Lemma 1 from~\cite[p.~148]{GL79}.
By Lemma~\ref{lem:connect_sum}, it suffices to show that we can either improve $\Sigma$ (while maintaining~\ref{2.1}--\ref{2.3}) and make it disjoint from $g(\Sigma)$, or we can produce a sphere $\Sigma'$ which bounds two knotted ball-arc pairs in $(S^3,K)$ such that $g(\Sigma')=\Sigma'$ and $\Sigma'$ meets $\kappa$ at exactly two points distinct from $v_1$ and $v_2$.

\begin{lemma}\label{lem:inessential}
Any curves of $\Sigma\cap g(\Sigma)$ which are inessential in $\Sigma-K$ can be removed without introducing new intersections.
\end{lemma}
\begin{proof}
Note that a curve in $\Sigma\cap g(\Sigma)$ is essential in $\Sigma-K$ if and only if it is essential in $g(\Sigma)-K$.
%Now suppose that there is a component of $\Sigma\cap g(\Sigma)$ that is inessential in $\Sigma-K$.
%Then there is a curve $c\in \Sigma\cap g(\Sigma)$ which is innermost in $g(\Sigma)-K$ as well as inessential, and $c$ bounds disks $D_1\subset g(\Sigma)-K$ and $D_2\subset \Sigma-K$.
Consider a component $c$ of $\Sigma\cap g(\Sigma)$ that is inessential in $\Sigma-K$ and is innermost in $g(\Sigma)-K$.
Then $c$ bounds closed disks $D_1\subset g(\Sigma)-K$ and $D_2\subset \Sigma-K$ and $D_1\cup D_2$ is an embedded 2-sphere.
As $D_1$ and $D_2$ are both disjoint from $K$, $D_1\cup D_2$ bounds a ball $B$ disjoint from $K$.

Case 1: $D_1\cap g(D_1)=\emptyset$.
Then, there is a neighborhood $N$ of $D_1$ such that $N\cap g(N)=\emptyset$ and $N\cap \Sigma\cap g(\Sigma)=c$. Improve $\Sigma$ by pushing $D_2$ past $D_1$ into $N$ using $B$. 
Since the only part of $\Sigma$ that changed now lies in $N$ and $N$ is now disjoint from $\Sigma\cap g(\Sigma)$, there are no new intersections.

Case 2: $D_1\cap g(D_1)\neq\emptyset$.
Since $D_1$ is innermost, this means that $c=g(c)$ and $D_1=g(D_2)$.
%Then $D_1\cup D_2$ is a sphere that bounds two 3-balls in $S^3$.
%One of the balls bounded by $D_1\cup D_2$ is disjoint from $K$, 
%so exchanging $D_1$ and $D_2$ removes $c$ without adding new intersections or changing the isotopy type of $\Sigma-K$.
Using $B$, push $D_2$ past $D_1$ to a parallel copy of $D_1$.
This removes $c$ without adding new intersections.
\end{proof}

%Let $c_1,c_2,\dots,c_n$ be the components of $\Sigma\cap g(\Sigma)$ listed in order such that $c_i$ is adjacent to $c_{i+1}$ for $i=1,2,\dots,n-1$. Let $A_{ij}\subset\Sigma$ be the annulus with $\partial A_{ij}=c_i\cup c_j$.
%Let $\pi\in\mathrm{Sym}(n)$ such that $g(c_i)=c_{\pi(i)}$.

By Lemma~\ref{lem:inessential}, we may assume all components of $\Sigma\cap g(\Sigma)$ are essential in $\Sigma-K$.
Let $c_1,c_2,\dots,c_n$ be the components of $\Sigma\cap g(\Sigma)$ and let $A_{ij}\subset\Sigma$ be the annulus with $\partial A_{ij}=c_i\cup c_j$ for $i\neq j$.
We assume $c_i$ and $c_{i+1}$ are adjacent in $\Sigma-K$ for $1\leq i\leq n-1$.
Let $\pi\in\mathrm{Sym}(n)$ be the permutation such that $g(c_i)=c_{\pi(i)}$.
As $g$ is an involution, either $\pi$ is the identity or $\pi$ has order two.

\begin{lemma}\label{lem:fixed_scc}
If $\pi(i)=i$ and $c_i$ bounds a disk $D\subset\Sigma$ such that $\Int D\cap g(\Int D)=\emptyset$, then either $c_i$ can be removed without introducing intersections or $\theta$ is a nontrivial order-3 connected sum.
\end{lemma}
\begin{proof}
Suppose $D$ and $i$ are as indicated. Then, $D\cup g(D)$ is a sphere invariant under $g$.
Since $g$ is orientation preserving, $D\cup g(D)$ must have at least one fixed point under $g$; in fact, it must have two because $\kappa$ meets $D\cup g(D)$ transversely.
As $\Int D\cap g(\Int D)=\emptyset$, all such fixed points must lie in $c_i$.
As $\Sigma$ is transverse to $\kappa$, $c_i$ contains a finite number of fixed points under $g$.
Let $a\subset c_i$ be an arc intersecting $\kappa$ exactly at its endpoints, so $\kappa\cap a=\partial a$.
Then $a\cup g(a)$ is a simple closed curve, so $c_i=a\cup g(a)$ and thus $c_i$ cannot contain more than two fixed points.
Thus, if $D\cup g(D)$ bounds two knotted ball-arc pairs, then we are done by Lemma~\ref{lem:connect_sum}.
%$D\cup g(D)$ is an invariant sphere. The interiors of $D$ and $g(D)$ are disjoint from $\kappa$ and $c$ can hit $\kappa$ at only two points.

Otherwise, $D\cup g(D)$ bounds an unknotted ball-arc pair $(B,B\cap K)$.
Using $B$, we can push $D$ (and any other components of $\Sigma\cap B$) past $g(D)$ to remove $c_i$.
\end{proof}

Thus, we must show that as long as $\Sigma\cap g(\Sigma)\neq\emptyset$, there is a curve $c_i$ as in Lemma~\ref{lem:fixed_scc}.

\begin{lemma}\label{lem:torus}
Suppose $T\subset S^3$ is a torus such that $g(T)=T$ and $T\cap\kappa=\emptyset$.
Let $Z\subset S^3$ be the half bounded by $T$ containing $\kappa$.
If $c\subset T$ is essential in $T$, null-homotopic in $Z$, and equivariant, then it is invariant (setwise).
\end{lemma}
\begin{proof}
Suppose $T$, $Z$, and $c$ are as indicated.
Since $\kappa\subset Z$, it is clear that $g(Z)=Z$.
As $c$ is equivariant, null-homotopic in $Z$, and disjoint from $\kappa$, $c$ bounds an equivariant disk $D\subset Z$ by the equivariant Dehn Lemma~\cite{Edm86}.

Suppose, by way of contradiction, that $g(c)\neq c$.
Then, as $c$ and $D$ are equivariant, we have $c\cap g(c)=\emptyset$ and $D\cap g(D)=\emptyset$.
The set $Z-(D\cup g(D))$ has two components, the closures in $Z$ of these are $B_1$ and $B_2$.
Each of $B_1$ and $B_2$ is bounded by the disks $D$ and $g(D)$ as well as an annulus in $T$, so both $B_1$ and $B_2$ are 3-balls.
Now, $\kappa$ is contained in one of these balls.
Without loss of generality, $\kappa\subset B_1$.
So, both $B_1$ and $B_2$ must be fixed setwise by $g$.
As $B_1$ is fixed by $g$ and $g$ is orientation preserving, $g$ must have a fixed point on $\partial B_1\subset T\cup D\cup g(D)$.
But, the set of fixed points of $g$ is exactly $\kappa$, and $T$, $D$, and $g(D)$ are all disjoint from $\kappa$.
This is a contradiction, so $c$ must be invariant.
\end{proof}

In the following lemma, we take $i,j,k,l\in\cpa{1,2,\ldots,n}$.

\begin{lemma}\label{lem:there_is_fixed_scc}
If $\pi(j)>j$, then there is some $j<i<\pi(j)$ such that $\pi(i)=i$.
In particular, there is some $i$ such that $\pi(i)=i$.
\end{lemma}
\begin{proof}
Suppose, by way of contradiction, that $\pi(j)>j$ and there is no such $i$.
Choose $k$ such that:
\begin{enumerate}
\item\label{1} $j\leq k<\pi(k)\leq\pi(j)$.
\item $\pi(k)-k$ is minimal such that~\ref{1} is satisfied.
\end{enumerate}
If $k<l<\pi(k)$, then we cannot have $k<\pi(l)<\pi(k)$.
So, $A_{k\pi(k)}\cap g(A_{k\pi(k)})=c_k\cup c_{\pi(k)}$ and $T=A_{k\pi(k)}\cup g(A_{k\pi(k)})$ is a torus.
Furthermore, $T=g(T)$ and $T$ is disjoint from $\kappa$.
%Let $Z$ be the solid 3-manifold with boundary $T$ such that $\kappa\subset Z$.
As $\kappa$ is connected, one component of $S^3-T$ contains $\kappa$.
Let $Z$ be its closure in $S^3$.
Since $T$ is also disjoint from $K$ and $K\cup\kappa$ is connected, we have $K\subset Z$.
Let $m$ be minimal such that $c_m\subset T$.
Then, $c_m$ bounds a disk in $\Sigma$ with interior disjoint from $T$ which intersects $K$.
So, $c_m$ and all other curves in $T$ of the same homotopy type are null-homotopic in $Z$.
In particular, $c_k$ is essential in $T$, null-homotopic in $Z$, and equivariant.
Hence, $c_k$ is invariant by Lemma~\ref{lem:torus}. This implies that $\pi(k)=k$, a contradiction.

The second part is immediate as either $\pi(1)=1$, or $\pi(1)>1$ and there is $1<i<\pi(1)$ such that $\pi(i)=i$.
\end{proof}

Let $i\in\cpa{1,2,\ldots,n}$ be minimal such that $\pi(i)=i$ (this $i$ exists by Lemma~\ref{lem:there_is_fixed_scc}).
Let $D\subset\Sigma$ be a disk with $\partial D=c_i$ such that: (i) $D$ contains $c_{i-1}$ in case $i>1$, (ii) $D$ does not contain $c_{i+1}$ in case $i<n$, and (iii) $D$ is either disk in $\Sigma$ bounded by $c_1$ in case $i=1$.
We claim that $D\cap g(D)=c_i$. Suppose, by way of contradiction, that $c_k\subset D\cap g(D)$.
Then, $1\leq k <i$ and $1\leq \pi(k)<i$. If $\pi(k)=k$, then $k$ contradicts minimality of $i$.
If $k<\pi(k)$, then Lemma~\ref{lem:there_is_fixed_scc} yields $k<i'<\pi(k)<i$ such that $\pi(i')=i'$, which contradicts minimality of $i$.
If $\pi(k)<k$, then let $j=\pi(k)$. So, $j=\pi(k)<k=\pi(j)$ since $\pi^2$ is the identity.
Again, Lemma~\ref{lem:there_is_fixed_scc} yields $j<i'<\pi(j)<i$ such that $\pi(i')=i'$, which contradicts minimality of $i$.
Hence, the claim $D\cap g(D)=c_i$ holds.
By Lemma~\ref{lem:fixed_scc}, either $c_i$ can be removed or $\theta$ is a nontrivial order-3 connected sum. Thus, if $\theta$ is not a nontrivial order-3 connected sum, then $\Sigma$ can be made disjoint from $g(\Sigma)$ and $\theta$ is a nontrivial order-2 connected sum by Lemma~\ref{lem:connect_sum}.
This completes the proof of the Main Theorem.

\section{Application}
\label{sec:applications}
To employ the Main Theorem, we must produce the lifted knot for a given theta-curve.
Fortunately, this is not difficult.
Given a theta-curve $\theta=\kappa\cup e$ with unknotted constituent knot $\kappa$, draw the lifted knot using the following algorithm:
\begin{enumerate}
\item Draw $\theta$ in $\R^3\cup\cpa{\infty}$ so that $\kappa$ comprises the $z$-axis and the point at infinity.
Consider the diagram given by projecting $\theta$ onto the $zy$-plane.
\item By an ambient isotopy fixing $\kappa$, arrange for all self-crossings of $e$ to have positive $y$-coordinate.
By a further isotopy, we may assume the diagram appears as shown in Figure~\ref{fig:ltc2d_1}.
\begin{figure}[htbp!]
\centering
\subfigure[\emph{Theta-curve $\theta$.}]
{
    \label{fig:ltc2d_1}
    \includegraphics[scale=1.0]{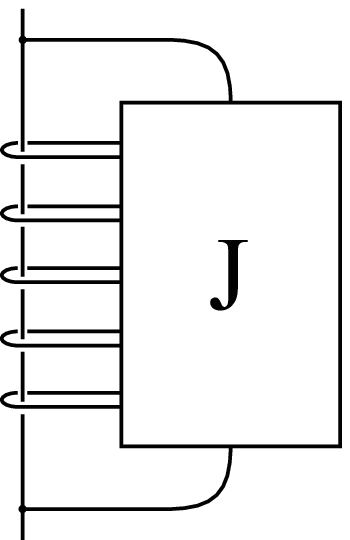}
}
\hspace{1cm}
\subfigure[\emph{Lift of $\theta$ to the double branch cover.}]
{
    \label{fig:ltc2d_2}
    \includegraphics[scale=1.0]{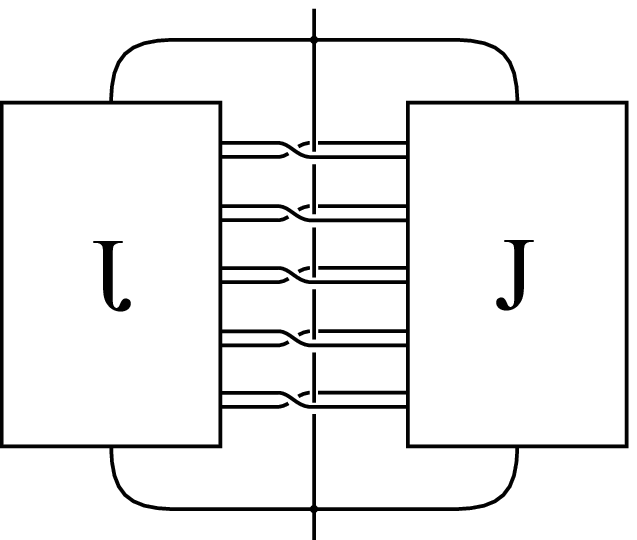}
}
\caption{Lifting a theta-curve to the double branched cover, branched over the unknotted constituent knot $\kappa$ pictured as the $z$-axis and the point at infinity.}
\label{fig:ltc2d}
\end{figure}
\item Now, lifting $e$ yields the knot $K$ given as the union of the following: (i) the diagam $J$, (ii) a copy of $J$ rotated one half of a revolution about the $z$-axis, and (iii) arcs between (i) and (ii) as shown in Figure~\ref{fig:ltc2d_2}.
\end{enumerate}

\begin{example}
Kinoshita's well-known theta-curve~\cite{Kin72} is shown in Figure~\ref{fig:kino_1}.
All three of its constituent knots are unknotted.
\begin{figure}[h!]
\centering
\subfigure[\emph{Kinoshita's theta-curve.}]
{
    \label{fig:kino_1}
    \includegraphics[scale=1.0]{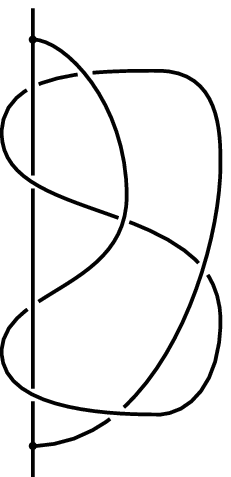}
}
\hspace{1.5cm}
\subfigure[\emph{Lift knot $K$ of Kinoshita's theta-curve.}]
{
    \label{fig:kino_2}
    \includegraphics[scale=1.0]{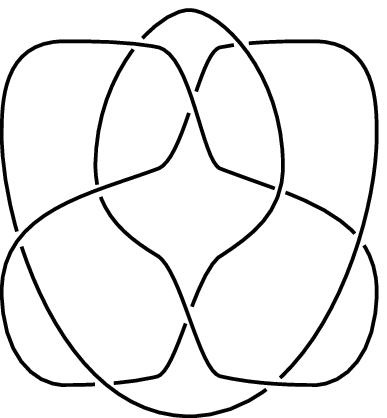}
}
%\captionsetup{labelformat=ss,textfont=it}
\caption{Lifting Kinoshita's theta-curve.}
\label{fig:kino}
\end{figure}
Applying the algorithm to Kinoshita's theta-curve, we obtain the lifted knot $K$ in Figure~\ref{fig:kino_2}.
With $K$ exhibited as a positive $3$-braid, it is a simple exercise to isotop $K$ to the standard $(3,5)$-torus knot (this fact was also observed by Wolcott~\cite{Wol87}).
As torus knots are prime~\cite[p.~95]{BZ03}, $K$ is prime and the Main Theorem implies that Kinoshita's theta-curve is prime as well.
\end{example}

\begin{remark}
Previously, the authors~\cite{CMBRS14} produced uncountably many isotopically distinct unions of three rays in $\R^3$ with the Brunnian property (namely, all three rays are knotted, but any two of them are unknotted).
To achieve this, we used sequences of three-component tangles lying in thickened spheres.
Our main tangle $A$ is shown in Figure~\ref{fig:block_A}.
\begin{figure}[htbp]
   \centering
   \includegraphics[scale=1.0]{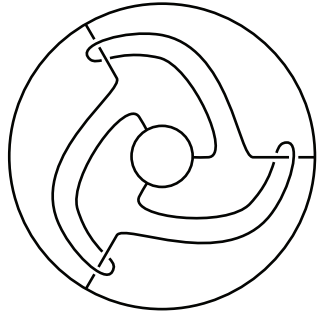}
   \caption{Tangle $A$ in a thickened sphere.}
   \label{fig:block_A}
\end{figure}
We discovered this tangle independently as described in~\cite{CMBRS14}.
A key property of the tangle $A$ proved in~\cite{CMBRS14} was that $A$ is irreducible (namely, no sphere separates $A$ into two nontrivial tangles).
By taking the thickened sphere containing $A$ and crushing each of the boundary spheres to a point, one obtains Kinoshita's theta-curve. As we have just observed, Kinoshita's theta-curve is prime. This immediately implies that the tangle $A$ is irreducible and provides an alternate proof of~\cite[Theorem 6.1]{CMBRS14}.
\end{remark}


\begin{thebibliography}{99}

%\bibitem[Bre72]{Bre72}
%	Glen E. Bredon,
%	\emph{Introduction to compact transformation groups},
%	Pure and Applied Mathematics \textbf{46},
%	Academic Press, New York-London, 1972.

\bibitem[BZ03]{BZ03}
	Gerhard Burde and Heiner Zieschang,
	\emph{Knots},
	de Gruyter Studies in Mathematics, vol. 5, second edition, Walter de Gruyter \& Co., Berlin,
	2003.

\bibitem[CMBRS14]{CMBRS14}
	Jack S. Calcut, Jules R. Metcalf-Burton, Taylor J. Richard, and Liam T. Solus,
	\emph{Borromean rays and hyperplanes},
	J. Knot Theory Ramifications \textbf{23} (2014), 46 pp.

\bibitem[Edm86]{Edm86}
	Allan L. Edmonds,
	\emph{A topological proof of the equivariant Dehn lemma},
	Trans. Amer. Math. Soc. \textbf{297} (1986), 605--615.

%\bibitem[Gor60]{Gor60}
%	C. McA. Gordon
%	\emph{3-Dimensional Topology up to 1960},
%	Elsevier Science, 1999.

\bibitem[GL79]{GL79}
	C. McA. Gordon and R. A. Litherland,
	\emph{Incompressible surfaces in branched coverings},
	in \emph{The Smith Conjecture},
	New York, 1979, 139--152.

\bibitem[KT80]{KT80}
	Paik Kee Kim and Jeffrey L. Tollefson,
	\emph{Splitting the PL involutions of nonprime $3$-manifolds},
	Michigan Math. J. \textbf{27} (1980), 259--274.

\bibitem[Kin72]{Kin72}
   Shin'ichi Kinoshita,
   \emph{On elementary ideals of polyhedra in the $3$-sphere},
   Pacific J. Math. \textbf{42} (1972), 89--98.

%\bibitem[Nor82]{Nor82}
%	F.~H.~Norwood,
%	\emph{Every two-generator knot is prime},
%	Proc. Amer. Math. Soc. \textbf{86} (1982), 143--147.

%\bibitem[Lic97]{Lic97}
%	W.B. Raymond Lickorish,
%	\emph{An Introduction to Knot Theory},
%	Springer-Verlag,
%	New York,
%	1997.

%\bibitem[MY81]{MY81}
%	William H. Meeks and Shing-Tung Yau
%	\emph{The equivariant Dehn's Lemma and loop theorem},
%	Birkh\"aiuser Verlag, Basel, 1981.

\bibitem[Mor04]{Mor04}
	Hiromasa Moriuchi,
	\emph{An enumeration of theta-curves with up to seven crossings},
	Proceedings of the East Asian School of Knots, Links, and Related Topics, February 16--20, 2004, 171--185.

\bibitem[Mor09]{Mor09}
	---------,
	\emph{An enumeration of theta-curves with up to seven crossings},
	J. Knot Theory Ramifications \textbf{18} (2009), 167--197.
	
%\bibitem[Rol90]{Rol90}
%	Dale Rolfsen,
%	\emph{Knots and links},
%	Mathematics Lecture Series \textbf{7},
%	Corrected reprint of the 1976 original,
%	Publish or Perish, Inc., Houston, TX, 1990.
	
%\bibitem[SnapPy]{SnapPy}
%	Marc Culler, Nathan M. Dunfield, and Jeffrey R. Weeks,
%	\emph{Snap{P}y, a computer program for studying the topology of $3$-manifolds},
%	Available at \url{http://snappy.computop.org} (26/02/2016).

%\bibitem[Thu97]{thurston}
%    W.~P.~Thurston,
%    \emph{Three-dimensional geometry and topology. Vol. 1},
%    Edited by Silvio Levy, Princeton University Press, Princeton, NJ, 1997.
	
\bibitem[Wol87]{Wol87}
	Keith Wolcott,
	\emph{The knotting of theta curves and other graphs in $S^3$},
	Geometry and topology (Athens, Ga., 1985),
	Lecture Notes in Pure and Appl. Math., vol. 105,
	Dekker, New York, 1987, 325--346.

\end{thebibliography}
\end{document}